\documentclass[11pt]{article}
\usepackage{amsmath,amsthm,amsfonts,amssymb,latexsym}
\usepackage[margin=1.2in, top=1in]{geometry}

\usepackage{hyperref}
\usepackage{comment}
\usepackage{enumerate}
\usepackage{color}
\usepackage[dvipsnames]{xcolor}

\usepackage[shortlabels]{enumitem}

\usepackage[english]{babel}

\usepackage{sectsty}

\allsectionsfont{\centering}

\makeatletter
\def\blfootnote{\gdef\@thefnmark{}\@footnotetext}
\makeatother

\usepackage{parskip}

\makeatletter
\def\thm@space@setup{%
  \thm@preskip=\parskip \thm@postskip=0pt
}
\makeatother

\numberwithin{equation}{section}

\begin{document}

\theoremstyle{plain}

\newtheorem{thm}{Theorem}[section]

\newtheorem{lem}[thm]{Lemma}
\newtheorem{Problem B}[thm]{Problem B}

\newtheorem{hyp}[thm]{Hypothesis}

\newtheorem{pro}[thm]{Proposition}
\newtheorem{conj}[thm]{Conjecture}
\newtheorem{cor}[thm]{Corollary}
\newtheorem{que}[thm]{Question}
\newtheorem{rem}[thm]{Remark}
\newtheorem{defi}[thm]{Definition}

\newtheorem*{thmA}{Theorem A}
\newtheorem*{thmB}{Theorem B}
\newtheorem*{corB}{Corollary B}
\newtheorem*{thmC}{Theorem C}
\newtheorem*{thmD}{Theorem D}
\newtheorem*{thmE}{Theorem E}

\newtheorem{thml}{Theorem}
\renewcommand*{\thethml}{\Alph{thml}} 
 
\newtheorem*{thmAcl}{Main Theorem$^{*}$}
\newtheorem*{thmBcl}{Theorem B$^{*}$}
\newcommand{\dd}{\mathrm{d}}

\newcommand{\wh}[1]{\widehat{#1}}

\newcommand{\Maxn}{\operatorname{Max_{\textbf{N}}}}
\newcommand{\Syl}{\operatorname{Syl}}
\newcommand{\Lin}{\operatorname{Lin}}
\newcommand{\U}{\mathbf{U}}
\newcommand{\R}{\mathbf{R}}
\newcommand{\dl}{\operatorname{dl}}
\newcommand{\Con}{\operatorname{Con}}
\newcommand{\cl}{\operatorname{cl}}
\newcommand{\Stab}{\operatorname{Stab}}
\newcommand{\Aut}{\operatorname{Aut}}
\newcommand{\Ker}{\operatorname{Ker}}
\newcommand{\InnDiag}{\operatorname{InnDiag}}
\newcommand{\fl}{\operatorname{fl}}
\newcommand{\Irr}{\operatorname{Irr}}
\newcommand{\FF}{\mathbb{F}}
\newcommand{\EE}{\mathbb{E}}
\newcommand{\normal}{\trianglelefteq}
\newcommand{\sn}{\normal\normal}
\newcommand{\Bl}{\mathrm{Bl}}
\newcommand{\bl}{\mathrm{bl}}
\newcommand{\h}{\mathrm{ht}}
\newcommand{\NN}{\mathbb{N}}
\newcommand{\N}{\mathbf{N}}
\newcommand{\bfC}{\mathbf{C}}
\newcommand{\bfO}{\mathbf{O}}
\newcommand{\bfF}{\mathbf{F}}
\def\GGG{{\mathcal G}}
\def\HHH{{\mathcal H}}
\def\HH{{\mathcal H}}
\def\irra#1#2{{\rm Irr}_{#1}(#2)}

\def\reg{{\rm reg}}

\renewcommand{\labelenumi}{\upshape (\roman{enumi})}

\newcommand{\PSL}{\operatorname{PSL}}
\newcommand{\PSU}{\operatorname{PSU}}
\newcommand{\alt}{\operatorname{Alt}}

\providecommand{\V}{\mathrm{V}}
\providecommand{\E}{\mathrm{E}}
\providecommand{\ir}{\mathrm{Irm_{rv}}}
\providecommand{\Irrr}{\mathrm{Irm_{rv}}}
\providecommand{\re}{\mathrm{Re}}

\numberwithin{equation}{section}
\def\irrp#1{{\rm Irr}_{p'}(#1)}

\def\ibrrp#1{{\rm IBr}_{\Bbb R, p'}(#1)}
\def\C{{\mathbb C}}
\def\Q{{\mathbb Q}}
\def\irr#1{{\rm Irr}(#1)}
\def\irrp#1{{\rm Irr}_{p^\prime}(#1)}
\def\irrq#1{{\rm Irr}_{q^\prime}(#1)}
\def \c#1{{\cal #1}}
\def \aut#1{{\rm Aut}(#1)}
\def\cent#1#2{{\bf C}_{#1}(#2)}
\def\norm#1#2{{\bf N}_{#1}(#2)}
\def\zent#1{{\bf Z}(#1)}
\def\syl#1#2{{\rm Syl}_#1(#2)}
\def\normal{\triangleleft\,}
\def\oh#1#2{{\bf O}_{#1}(#2)}
\def\Oh#1#2{{\bf O}^{#1}(#2)}
\def\det#1{{\rm det}(#1)}
\def\gal#1{{\rm Gal}(#1)}
\def\ker#1{{\rm ker}(#1)}
\def\normalm#1#2{{\bf N}_{#1}(#2)}
\def\alt#1{{\rm Alt}(#1)}
\def\iitem#1{\goodbreak\par\noindent{\bf #1}}
   \def \mod#1{\, {\rm mod} \, #1 \, }
\def\sbs{\subseteq}

\def\gc{{\bf GC}}
\def\ngc{{non-{\bf GC}}}
\def\ngcs{{non-{\bf GC}$^*$}}
\newcommand{\notd}{{\!\not{|}}}

\newcommand{\Z}{\mathbf{Z}}
\newcommand{\miquelcomment}{\textcolor{blue}}
\newcommand{\damcomment}{\textcolor{red}}
\newcommand{\Out}{{\mathrm {Out}}}
\newcommand{\Mult}{{\mathrm {Mult}}}
\newcommand{\Inn}{{\mathrm {Inn}}}
\newcommand{\Fong}{{\mathrm{Fong}}}
\newcommand{\IBR}{{\mathrm {IBr}}}
\newcommand{\IBRL}{{\mathrm {IBr}}_{\ell}}
\newcommand{\IBRP}{{\mathrm {IBr}}_{p}}
\newcommand{\cd}{\mathrm{cd}}
\newcommand{\ord}{{\mathrm {ord}}}
\def\id{\mathop{\mathrm{ id}}\nolimits}
\renewcommand{\Im}{{\mathrm {Im}}}
\newcommand{\Ind}{{\mathrm {Ind}}}
\newcommand{\diag}{{\mathrm {diag}}}
\newcommand{\soc}{{\mathrm {soc}}}
\newcommand{\End}{{\mathrm {End}}}
\newcommand{\sol}{{\mathrm {sol}}}
\newcommand{\Hom}{{\mathrm {Hom}}}
\newcommand{\Mor}{{\mathrm {Mor}}}
\newcommand{\Mat}{{\mathrm {Mat}}}
\def\rank{\mathop{\mathrm{ rank}}\nolimits}
\newcommand{\Tr}{{\mathrm {Tr}}}
\newcommand{\tr}{{\mathrm {tr}}}
\newcommand{\Gal}{{\rm Gal}}
\newcommand{\Spec}{{\mathrm {Spec}}}
\newcommand{\ad}{{\mathrm {ad}}}
\newcommand{\Sym}{{\mathrm {Sym}}}
\newcommand{\Char}{{\mathrm {Char}}}
\newcommand{\pr}{{\mathrm {pr}}}
\newcommand{\rad}{{\mathrm {rad}}}
\newcommand{\abel}{{\mathrm {abel}}}
\newcommand{\PGL}{{\mathrm {PGL}}}
\newcommand{\PCSp}{{\mathrm {PCSp}}}
\newcommand{\PGU}{{\mathrm {PGU}}}
\newcommand{\codim}{{\mathrm {codim}}}
\newcommand{\ind}{{\mathrm {ind}}}
\newcommand{\Res}{{\mathrm {Res}}}
\newcommand{\Lie}{{\mathrm {Lie}}}
\newcommand{\Ext}{{\mathrm {Ext}}}
\newcommand{\Alt}{{\mathrm {Alt}}}
\newcommand{\AAA}{{\sf A}}
\newcommand{\SSS}{{\sf S}}
\newcommand{\DDD}{{\sf D}}
\newcommand{\QQQ}{{\sf Q}}
\newcommand{\CCC}{{\sf C}}
\newcommand{\SL}{{\mathrm {SL}}}
\newcommand{\Sp}{{\mathrm {Sp}}}
\newcommand{\PSp}{{\mathrm {PSp}}}
\newcommand{\SU}{{\mathrm {SU}}}
\newcommand{\GL}{{\mathrm {GL}}}
\newcommand{\GU}{{\mathrm {GU}}}
\newcommand{\Spin}{{\mathrm {Spin}}}
\newcommand{\CC}{{\mathbb C}}
\newcommand{\CB}{{\mathbf C}}
\newcommand{\RR}{{\mathbb R}}
\newcommand{\QQ}{{\mathbb Q}}
\newcommand{\ZZ}{{\mathbb Z}}
\newcommand{\bfN}{{\mathbf N}}
\newcommand{\bfZ}{{\mathbf Z}}
\newcommand{\PP}{{\mathbb P}}
\newcommand{\cG}{{\mathcal G}}
\newcommand{\cH}{{\mathcal H}}
\newcommand{\cQ}{{\mathcal Q}}
\newcommand{\GA}{{\mathfrak G}}
\newcommand{\cT}{{\mathcal T}}
\newcommand{\cL}{{\mathcal L}}
\newcommand{\IBr}{\mathrm{IBr}}
\newcommand{\cS}{{\mathcal S}}
\newcommand{\cR}{{\mathcal R}}
\newcommand{\GCD}{\GC^{*}}
\newcommand{\TCD}{\TC^{*}}
\newcommand{\FD}{F^{*}}
\newcommand{\GD}{G^{*}}
\newcommand{\HD}{H^{*}}
\newcommand{\GCF}{\GC^{F}}
\newcommand{\TCF}{\TC^{F}}
\newcommand{\PCF}{\PC^{F}}
\newcommand{\GCDF}{(\GC^{*})^{F^{*}}}
\newcommand{\RGTT}{R^{\GC}_{\TC}(\vartheta)}
\newcommand{\RGTA}{R^{\GC}_{\TC}(1)}
\newcommand{\Om}{\Omega}
\newcommand{\eps}{\epsilon}
\newcommand{\varep}{\varepsilon}
\newcommand{\al}{\alpha}
\newcommand{\chis}{\chi_{s}}
\newcommand{\sigmad}{\sigma^{*}}
\newcommand{\PA}{\boldsymbol{\alpha}}
\newcommand{\gam}{\gamma}
\newcommand{\lam}{\lambda}
\newcommand{\la}{\langle}
\newcommand{\genf}{F^*}
\newcommand{\ra}{\rangle}
\newcommand{\hs}{\hat{s}}
\newcommand{\htt}{\hat{t}}
\newcommand{\tG}{\hat G}
\newcommand{\St}{\mathsf {St}}
\newcommand{\bfs}{\boldsymbol{s}}
\newcommand{\bfl}{\boldsymbol{\lambda}}
\newcommand{\tn}{\hspace{0.5mm}^{t}\hspace*{-0.2mm}}
\newcommand{\ta}{\hspace{0.5mm}^{2}\hspace*{-0.2mm}}
\newcommand{\tb}{\hspace{0.5mm}^{3}\hspace*{-0.2mm}}
\def\skipa{\vspace{-1.5mm} & \vspace{-1.5mm} & \vspace{-1.5mm}\\}
\newcommand{\tw}[1]{{}^#1\!}
\renewcommand{\mod}{\bmod \,}

\marginparsep-0.5cm

\renewcommand{\thefootnote}{\fnsymbol{footnote}}
\footnotesep6.5pt

\selectlanguage{english}

\title{\LARGE
{\bf Degree divisibility in Alperin--McKay correspondences}
\author{
J. Miquel Mart\'inez
\\
\and
Damiano Rossi
}
\date{}
\blfootnote{\emph{$2010$ Mathematical Subject Classification:} $20$C$15$, $20$C$20$.
\\
\emph{Key words and phrases:} Alperin--McKay conjecture, Brauer correspondence, character degrees.
\\
The authors would like to thank Gabriel Navarro for providing a proof of Proposition \ref{gabriel} and for helpful conversations on the topic of this paper. Furthermore, we are thankful to Radha Kessar and Markus Linckelmann for the permission to include Proposition \ref{prop:Kessar-Linckelmann}. They also thank Noelia Rizo for a thorough read of an earlier version of this manuscript. Part of these results were obtained during the second author's doctoral studies at the Bergische Universit\"at Wuppertal funded by the research training group \textit{GRK2240: Algebro-geometric Methods in Algebra, Arithmetic and Topology} of the DFG. The first author acknowledges support from Ministerio de Ciencia e Innovaci\'on PID2019-103854GB-I00, Generalitat Valenciana CIAICO/2021-163, as well as a fellowship UV-INV-PREDOC20-1356056 from Universitat de Val\`encia. The second author is supported by the EPSRC grant EP/T004592/1.}
}
\maketitle

\begin{abstract}
Let $p$ be a prime, $B$ a $p$-block of a finite group $G$ and $b$ its Brauer correspondent. According to the Alperin--McKay Conjecture, there exists a bijection between the set of irreducible ordinary characters of height zero of $B$ and those of $b$. In this paper, we show that whenever $G$ is $p$-solvable such a bijection can be found, both for ordinary and Brauer characters, with the additional property of being compatible with divisibility of character degrees. In this case, we also show that the dimension of $b$ divides the dimension of $B$.
\end{abstract}

\section*{Introduction}

The McKay Conjecture is currently one of the main open problems in representation theory of finite groups. It states that, if $p$ is a prime number, $G$ is a finite group and $P$ is a Sylow $p$-subgroup of $G$, then 
\[|\Irr_{p'}(G)|=|\Irr_{p'}(\norm G P)|\]
where, for any finite group $H$, we denote by $\Irr_{p'}(H)$ the set of irreducible complex characters of degree not divisible by $p$.

For a solvable group $G$, Turull showed that the character correspondence from the McKay conjecture can be chosen to be compatible with divisibility of character degrees. Namely, in \cite{T07} he showed that there exists a bijection
\[\Omega:\Irr_{p'}(G)\to\Irr_{p'}(\norm G P)\]
such that $\Omega(\chi)(1)$ divides $\chi(1)$ for every $\chi\in\Irr_{p'}(G)$. Subsequently, in \cite{R19} Rizo extended Turull's result to $p$-solvable groups by assuming the fact that the Glauberman correspondence is compatible with divisibility of character degrees, a fact later proved by Geck \cite{Geck} and whose proof relies on \cite{HT94}. Moreover, \cite[Theorem 3.1]{BNRS22} provides a version of these results for $p$-Brauer characters.

In the 1970s, Alperin introduced an important generalisation of the the McKay Conjecture by considering $p$-blocks. The Alperin--McKay Conjecture posits that, if $B$ is a $p$-block of a finite group $G$ and $b$ is its Brauer correspondent, then
\[|\Irr_0(B)|=|\Irr_0(b)|\]
where, for any $p$-block $C$ of a finite group, we denote by $\Irr_0(C)$ the set of height zero characters of $C$. 

Our first theorem generalises all the above mentioned results by showing that, for both ordinary and $p$-Brauer characters of $p$-solvable groups, there exists a bijection between the height zero characters of $B$ and $b$ which is compatible with divisibility of character degrees.

\begin{thml}
\label{thm:AM with divisibility}
Let $G$ be a finite $p$-solvable group, $B$ a $p$-block of $G$ with defect group $D$, and consider its Brauer correspondent $b\in\Bl(\norm G D)$. Then there exist bijections
\[\Omega:\Irr_0(B)\to\Irr_0(b)\]
and
\[\Psi:\IBr_0(B)\to\IBr_0(b)\]
such that $\Omega(\chi)(1)$ divides $\chi(1)$ and $\Psi(\varphi)(1)$ divides $\varphi(1)$ for every $\chi\in\Irr_0(B)$ and $\varphi\in\IBr_0(B)$.
\end{thml}

A key ingredient in the proof of the above theorem is Proposition \ref{gabriel}, which is a generalisation of the main result of \cite{N03} and was kindly communicated to us by Navarro. We also mention that work of Navarro, Sp\"ath and Tiep shows that, if $G$ is $p$-solvable and the inductive McKay condition holds for the prime $q$, then there exists a bijection between $p$-Brauer characters of $q'$-degree of $G$ and $p$-Brauer characters of $q'$-degree of a $q$-Sylow normaliser (see \cite[Theorem C]{Nav-Spa-Tie17}). However, in this paper we do not consider this cross-characteristic case. Furthermore, at the end of Section \ref{sec:Thm A} we ask whether the bijections of Theorem \ref{thm:AM with divisibility} can be chosen to be compatible with decomposition numbers (see Question \ref{question}).

Next, we consider divisibility between the dimension of a $p$-block and that of its Brauer correspondent. For a $p$-block $C$, we define $\dim(C)$ to be the dimension of the corresponding block algebra. If $B$ is a $p$-block with Brauer correspondent $b$, then a theorem of Brauer shows that $\dim(b)_p$ divides $\dim(B)_p$ \cite[Theorem 10.1.1]{LinII}. In our next theorem, we show that divisibility holds for the full dimension whenever $G$ is $p$-solvable. 

\begin{thml}
\label{thm:Dimension divisibility}
Let $G$ be a finite $p$-solvable group, $B$ a $p$-block of $G$ with defect group $D$ and $b\in\Bl(\norm G D)$ its Brauer correspondent. Then $\dim(b)$ divides $\dim(B)$.
\end{thml}

Notice that in general, if $H$ is a subgroup of $G$ and $c$ is a $p$-block of $H$ that induces the $p$-block $C$ of $G$, then it is not clear even whether $\dim(c)\leq \dim(C)$. However, if $H$ is normal in $G$, then Kessar and Linckelmann recently proved that divisibility holds with respect to covering of blocks, hence answering a question posed by Navarro. With their kind permission we include this result in Proposition \ref{prop:Kessar-Linckelmann}.

The results of Theorem \ref{thm:AM with divisibility} and Theorem \ref{thm:Dimension divisibility} cannot be extended to arbitrary finite groups. As a counterexample consider the principal $2$-block of the alternating group $A_5$.

The paper is structured as follows. In Section \ref{sec:Preliminaries} we introduce some preliminary results including an extension of Fong's theory of $p$-solvable groups and Navarro's theorem. 
Then, in Section \ref{sec:Thm A} we prove Theorem \ref{thm:AM with divisibility} in a series of steps working by induction on the order of the group. Finally, in Section \ref{sec:Thm B} we prove Theorem \ref{thm:Dimension divisibility} and include a proof of the result of Kessar--Linckelmann on divisibility of block dimensions with respect to normal subgroups.

\section{Preliminary results}
\label{sec:Preliminaries}

Throughout the paper we use basic notation and terminology on character theory and $p$-blocks as introduced in \cite{Isa76}, \cite{N98} and \cite{N18}. In this section we collect some useful results that will be used in the subsequent sections.

To start, we consider some preliminary results on the Glauberman correspondence. If $Q$ is a group acting on $L$ with $(|L|,|Q|)=1$ and $\vartheta\in\Irr(L)$ is $Q$-invariant, then we denote by $f_Q(\vartheta)\in\Irr(\cent L Q)$ its Glauberman correspondent. In \cite{T08}, Turull proved the existence of bijections above the Glauberman correspondence by constructing certain isomorphisms of character triples. This important result has been extended in various context and provides a fundamental tool for the understanding of the local-global conjectures (see, for instance, \cite[Theorem 5.13]{Nav-Spa14}, \cite{Lad16}, \cite[Section 4]{Nav-Spa-Val} and \cite[Theorem 3.7]{Ros22}). In this paper, we provide yet another consequence of Turull's result and show that these bijections above the Glauberman correspondence are also compatible with divisibility of character triples. In addition, we obtain similar bijections for Brauer characters by applying ideas of Laradji \cite{Lar14}.

\begin{pro}
\label{prop:Above Glauberman with divisibility}
Let $L$ be a normal $p'$-subgroup of $G$ and consider a $p$-subgroup $Q$ of $G$ such that $N:=LQ\unlhd G$. Let $\vartheta$ be a $G$-invariant irreducible character of $L$ and consider the $Q$-Glauberman correspondent $f_Q(\vartheta)\in\Irr(\cent L Q)$ of $\vartheta$. Then there exist bijections
\[\Omega_Q:\Irr(G\mid \vartheta)\to\Irr(\norm G Q\mid f_Q(\vartheta))\]
and
\[\Psi_Q:\IBr(G\mid \vartheta)\to\IBr(\norm G Q\mid f_Q(\vartheta))\]
such that $\Omega_Q(\chi)(1)$ divides $\chi(1)$ and $\Psi_Q(\varphi)(1)$ divides $\varphi(1)$ for every $\chi\in\Irr(G\mid \vartheta)$ and $\varphi\in\IBr(G\mid \vartheta)$.
\end{pro}

\begin{proof}
By \cite[Theorem 6.5]{T08} and \cite[Theorem 7.12]{T09} the character triple $(G,L,\vartheta)$ is isomorphic to $(\norm G Q,\cent L Q, f_Q(\vartheta))$. Applying \cite[Lemma 11.24]{Isa76} and \cite[Lemma 2.3]{Lar14} we obtain bijections $\Omega_Q:\Irr(G\mid \vartheta)\to\Irr(\norm G Q\mid f_Q(\vartheta))$ and $\Psi_Q:\IBr(G\mid \vartheta)\to\IBr(\norm G Q\mid f_Q(\vartheta))$ such that $\Omega_Q(\chi)(1)/f_Q(\vartheta)(1)=\chi(1)/\vartheta(1)$ and $\Psi_Q(\varphi)(1)/f_Q(\vartheta)(1)=\varphi(1)/\vartheta(1)$ for every $\chi\in\Irr(G\mid \vartheta)$ and $\varphi\in\IBr(G\mid \vartheta)$. Since $f_Q(\vartheta)(1)$ divides $\vartheta(1)$ by \cite{Geck}, we conclude that $\Omega_Q(\chi)(1)$ divides $\chi(1)$ and that $\Psi_Q(\varphi)(1)$ divides $\varphi(1)$.
\end{proof}

\begin{rem}
\label{rem}
Observe that by the proof of Proposition \ref{prop:Above Glauberman with divisibility} it follows that the bijections $\Omega_Q$ and $\Psi_Q$ restrict to the subsets of $p'$-degree characters.
\end{rem}

Next, we consider a famous theorem of Fong \cite[Theorem 10.20]{N98}. If $G$ is $p$-solvable and $\vartheta\in\Irr(\oh{p'}{G})$ is $G$-invariant, then there exists a block $B\in\Bl(G)$ whose defect groups are Sylow $p$-subgroups and such that $\Irr(B)=\Irr(G|\vartheta)$. Our next result is used to replace Fong's theorem in the case where we consider a normal $p'$-subgroup $L$ of $G$ (but not necessarily $\oh{p'}G$) and an irreducible character $\vartheta\in\Irr(L)$ which is not necessarily $G$-invariant.

Before proceeding with the proof of the next proposition, we introduce some terminology. Let $G$ be a finite group, $B$ a block of $G$ with defect group $D$ and $L\normal G$ a $p'$-subgroup of $G$. Then, we say that a character $\vartheta\in\Irr(L)$ is $(B,D)$-\textbf{good} if $\Irr(B)\sbs\Irr(G|\vartheta)$ and $D$ is a defect group of the Fong--Reynolds correspondent of $B$ over $\bl(\vartheta)$. Notice that if $L\normal G$ is a $p'$-subgroup, then such a character $\vartheta$ always exists and is determined up to $\norm G D$-conjugation. In the following proof, we denote by $\h(\chi)$ the height of an irreducible ordinary or Brauer character.

\begin{pro}
\label{fongreynolds}
Let $G$ be a finite $p$-solvable group, $B$ a $p$-block of $G$ with defect group $D$ and consider its Brauer correspondent $b\in\Bl(\norm G D)$. If $L\normal G$ is a $p'$-subgroup and $\vartheta\in\Irr(L)$ is $(B,D)$-good, then there exists a subgroup $L\leq U\leq G_\vartheta$ and a $U$-invariant character $\nu\in\Irr(\oh{p'}{U})$ such that
\begin{enumerate}
\item induction $\psi\mapsto\psi^G$ defines a bijection $\Irr(U\mid\nu)\to\Irr(B)$;
\item induction $\varphi\mapsto\varphi^G$ defines a bijection $\IBr(U\mid\nu)\mapsto\IBr(B)$;
\item in the above bijections, $\Irr_{p'}(U\mid\nu)$ and $\IBr_{p'}(U\mid\nu)$ maps onto $\Irr_0(B)$ and $\IBr_0(B)$ respectively;
\item $D\in\Syl_p(U)$ and $\nu$ lies over $\vartheta$;
\item induction $\psi\mapsto\psi^{\norm G D}$ defines a bijection $\Irr(\norm U D\mid f_D(\nu))\to\Irr(b)$;
\item induction $\varphi\mapsto\varphi^{\norm G D}$ defines a bijection $\IBr(\norm U D\mid f_D(\nu))\to\IBr(b)$;
\item in the above bijections, $\Irr_{p'}(\norm U D\mid f_D(\nu))$ and $\IBr_{p'}(\norm U D\mid f_D(\nu))$ maps onto $\Irr_0(b)$ and $\IBr_0(b)$ respectively.
\end{enumerate}
\end{pro}

\begin{proof}
We argue by induction on $|G:L|$. Let $B_\vartheta\in\Bl(G_\vartheta)$ be the Fong--Reynolds correspondent of $B$ over $\bl(\vartheta)$ \cite[Theorem 9.14]{N98} and suppose that $G_\vartheta\leq G$. Since $\vartheta$ is $(B_\vartheta,D)$-good, the inductive hypothesis yields a subgroup $U<G_\vartheta$ and a $U$-invariant character $\nu\in\Irr(\oh{p'}U)$ satisfying the properties (i)--(vii) with respect to $(B_\vartheta, D)$. Now, induction of characters defines bijections $\Irr(U\mid\nu)\to \Irr(B_\vartheta)$ and $\IBr(U\mid\nu)\to\IBr(B_\vartheta)$. On the other hand, the Fong--Reynolds correspondence shows that induction of characters also defines bijections $\Irr(B_\vartheta)\to\Irr(B)$ and $\IBr(B_\vartheta)\to \IBr(B)$ and hence we obtain (i) and (ii).

Next, the inductive hypothesis implies that $\Irr_{p'}(U\mid\nu)$ maps onto $\Irr_0(B_\vartheta)$, which maps onto $\Irr_0(B)$ by part (b) of \cite[Theorem 9.14]{N98}. Let $\varphi\in\IBr_0(B)$ and $\beta\in\IBr(B_\vartheta)$ such that $\beta^G=\varphi$. By the Fong--Swan theorem \cite[Theorem 10.1]{N98} there is some $\psi\in\Irr(B_\vartheta)$ such that $\psi^0=\beta$ and $\h(\beta)=\h(\psi)$, where $\psi^0$ denotes the restriction of $\psi$ to $p$-regular elements. Observing that $\psi^G\in\Irr(B)$ and $\varphi=\beta^G=(\psi^0)^G=(\psi^G)^0$, we deduce that $\psi^G\in\Irr_0(B)$ and hence $\psi\in\Irr_0(B_\vartheta)$ by part (b) of  \cite[Theorem 9.14]{N98}. It follows that $\h(\beta)=0$ and thus (iii) holds. 

Let $f_D(\vartheta)\in\Irr(\cent L D)$ and observe that $\norm G D_{f_D(\vartheta)} = \norm G D_\vartheta=\norm{G_\vartheta}D$. Let $b_\vartheta\in\norm{G_\vartheta}D$ be the Fong--Reynolds correspondent of $b$ over $\bl(f_D(\vartheta))$. According to Brauer's first main theorem the induced block $(b_\vartheta)^{G_\vartheta}$ is defined and therefore $(b_{\vartheta})^{G}=((b_{\vartheta})^{\norm G D})^G=b^G=B$ by using \cite[Problem 4.2]{N98}. Now, the Fong--Reynolds correspondence implies that $(b_\vartheta)^{G_\vartheta}=B_\vartheta$ and so $b_{\vartheta}$ is the Brauer correspondent of $B_\vartheta$. By inductive hypothesis, induction of characters defines bijections $\Irr(\norm U D\mid f_D(\nu))\to\Irr(b_\vartheta)$ and $\IBr(\norm U D\mid f_D(\nu))\to \IBr(b_\vartheta)$. Arguing as in the previous paragraphs we conclude that (v), (vi) and (vii) are also satisfied. This shows that it is no loss of generality to assume $G_\vartheta=G$.

Finally, set $N:=\oh{p'}G$. If $N=L$, then the result follows from \cite[Theorem 10.20]{N98} by choosing $U=G$ and $\nu=\vartheta$. Thus, we assume $L<N$. Let $\eta\in\Irr(N)$ be $(B,D)$-good and notice that, since $\vartheta$ is $G$-invariant, Clifford's theorem implies that $\eta$ lies over $\vartheta$. Since $|G:N|<|G:L|$, the inductive hypothesis yields a subgroup $N\leq U\leq G_\eta$ and a $U$-invariant character $\nu\in\Irr(\oh{p'}U)$ satisfying the properties (i)-(vii) with respect to $N$ and $\eta$. Now, the result follows by noticing that $L\leq U\leq G_\vartheta$ and that $\nu$ covers $\vartheta$.
\end{proof}

The last result of this section is a generalization of the main result of \cite{N03} whose proof has been kindly provided to us by Navarro.

\begin{pro}[Navarro]
\label{gabriel}
Let $G$ be a $p$-solvable group, $U\leq G$ and consider a $p$-subgroup $P$ of $G$ such that $P\cap U\in\Syl_p(U)$. Then $|U:\norm U {P\cap U}|$ divides $|G:\norm G P|$.
\end{pro}

\begin{proof}
We argue by induction, first on $|G:U|$, and then on $|G|$. Set $K:=\oh p G$ and $Q:=P\cap U$. Suppose that $K$ is not contained in $U$ and define $U_0=UK$ and $P_0=PK$. Since $U\cap P\leq U\cap P_0$ and $U\cap P_0$ is a $p$-subgroup of $U$, we deduce that $U\cap P_0=U\cap P$. Then, by Dedekind's Lemma we conclude that $U_0\cap P_0\in\Syl_p(U_0)$. Noticing that $|G:U_0|<|G:U|$, the inductive hypothesis implies that $|U_0:\norm {U_0}{P_0\cap U_0}|$ divides $|G:\norm G {P_0}|$ which divides $|G:\norm G P|$. Now, the result follows since $|U_0:\norm {U_0}{P_0\cap U_0}|=|U:\norm U Q|$ and hence we may assume that $K\leq U$. In this case, we may assume that $K=1$ by working in $G/K$.

Since $G$ is $p$-solvable, it follows that $L:=\oh{p'}G>1$. Then, by applying the inductive hypothesis to $G/L$, we deduce that $|UL:\norm {UL}{QL}|$ divides $|G:\norm G P L|$. Because $|UL:\norm {UL}{QL}|=|U:\norm U Q(L\cap U)|$, it suffices to show that $|\norm U Q(L\cap U):\norm U Q|$ divides $|L\norm G P:\norm G P|$, or equivalently, that $|L\cap U:\cent{L\cap U}Q|$ divides $|L:\cent L P|$. By \cite[Lemma 2.1]{N03}, it follows that $|L\cap U:\cent {L\cap U}{Q}|$ divides $|L:\cent L Q|$ which divides $|L:\cent L P|$. This concludes the proof.
\end{proof}

Notice that the above proposition immediately implies the main result of \cite{N03}. In fact, if we chose $P\in\Syl_p(G)$ such that $P\cap U\in\Syl_p(U)$, then Proposition \ref{gabriel} shows that the number of Sylow $p$-subgroups of $U$, i.e. $|U:\norm{U}{U\cap P}|$, divides the number of Sylow $p$-subgroups of $G$, i.e. $|G:\norm G P|$.

Before proceeding further, we remark the the assumptions made in the above proposition are necessary. First, consider $G=A_5$, $U=A_4$ and $P\in\Syl_3(U)$. In this case, $|U:\norm U P|=4$ does not divide $|G:\norm G P|=10$. This shows that the result fails if $G$ is not $p$-solvable. Next, let $G=S_4$, $U\in\Syl_2(G)$ and consider a subgroup $P$ of order $2$ in $\oh{2}{G}$. Now $|G:\norm G P|=3$ is not divisible by $|U:\norm U P|$ since $P$ is not normal in $U$. Therefore, the hypothesis $P\cap U\in\Syl_p(U)$ is also necessary.

\section{Proof of Theorem \ref{thm:AM with divisibility}}
\label{sec:Thm A}

We now proceed to prove Theorem \ref{thm:AM with divisibility} by induction on the order of $G$. First, we show that without loss of generality we may assume $B$ to be a block of maximal defect.

\begin{lem}
\label{lem 1}
We may assume that $B$ covers a $G$-invariant $\vartheta\in\Irr(\oh{p'}{G})$. In particular, we have $D\in\Syl_p(G)$, $\Irr(B)=\Irr(G\mid\vartheta)$ and $\IBr(B)=\IBr(G\mid \vartheta)$.
\end{lem}

\begin{proof}
Let $L:=\oh{p'}{G}$ and consider a character $\vartheta\in\Irr(L)$ whose block is covered by $B$. Without loss of generality we may assume that $\vartheta$ is $(B,D)$-good. If $\vartheta$ is $G$-invariant, then the result follows from \cite[Theorem 10.20]{N98}. Thus, we assume $G_\vartheta<G$ and consider a subgroup $U\leq G_\vartheta$ and a $U$-invariant character $\nu\in\Irr(\oh{p'}{U})$ as in Proposition \ref{fongreynolds}. By \cite[Theorem 10.20]{N98} there exists a unique block $C$ of $U$ covering $\bl(\nu)$ and we have $\Irr(C)=\Irr(U\mid \nu)$, $\IBr(C)=\IBr(U\mid \nu)$ and $D$ is a defect group of $C$. Furthermore, if $c$ is the Brauer correspondent of $C$ in $\norm U D$ and $f_D(\nu)$ is the $D$-Glauberman correspondent of $\nu$, then $\Irr(c)=\Irr(\norm U D\mid f_D(\nu))$ and $\IBr(c)=\IBr(\norm U D\mid f_D(\nu))$ according to \cite[Theorem 0.29]{manzwolf}. Since $U<G$, the inductive hypothesis implies that there exist bijections $ \Omega_0:\Irr_{p'}(U\mid\nu)\to \Irr_{p'}(\norm U D\mid f_D(\nu))$ and $\Psi_0:\IBr_{p'}(U\mid\nu)\to \IBr_{p'}(\norm U D\mid f_D(\nu))$ such that $\Omega_0(\psi)(1)$ divides $\psi(1)$ and $\Psi_0(\eta)(1)$ divides $\eta(1)$ for every $\psi\in\Irr_{p'}(U\mid \nu)$ and every $\eta\in\IBr_{p'}(U\mid \nu)$. Next, by applying Proposition \ref{fongreynolds}, we obtain bijections $\Omega:\Irr_0(B)\to\Irr_0(b)$ and $\Psi:\IBr_0(B)\to\IBr_0(b)$ by setting
\[\Omega(\psi^G):=\Omega_0(\psi)^{\norm G D}\]
and
\[\Psi(\eta^G):=\Psi_0(\eta)^{\norm G D}\]
for every $\psi\in\Irr_{p'}(U\mid \nu)$ and every $\eta\in\IBr_{p'}(U\mid \nu)$. Furthermore, by Proposition \ref{gabriel} we conclude that $\Omega(\psi^G)(1)=|\norm G D:\norm U D|\Omega_0(\psi)(1)$ divides $|G:U|\psi(1)=\psi^G(1)$ and, similarly, that $\Psi(\eta^G)(1)$ divides $\eta^G(1)$.
\end{proof}

Next, we show that it is no loss of generality to assume $K:=\oh{p}{G}=1$.

\begin{lem}
\label{lem 2}
Let $\kappa\in\Irr(K)$ and suppose that $\Irr_0(B\mid \kappa)$ is non-empty. Then $\kappa$ is linear and extends to its stabiliser in $G$.
\end{lem}

\begin{proof}
First, observe that $B$ is a block of maximal defect by Lemma \ref{lem 1} and hence $\Irr_0(B)=\Irr_{p'}(B)$. In particular, the character $\kappa$ must be linear. Next, let $\chi\in\Irr_{p'}(B\mid\kappa)$ and denote by $\psi\in\Irr(G_\kappa)$ its Clifford correspondent over $\kappa$. Since $p$ does not divide the degree of $\chi$, it follows that $\psi\in\Irr_{p'}(G_\kappa)$ and that $G_\kappa$ contains some Sylow $p$-subgroup $P$ of $G$. Then, the restriction $\psi_P$ contains some linear constituent which must be an extension of $\kappa$. By \cite[Theorem 5.11]{N18}, we conclude that $\kappa$ extends to $G_\kappa$.
\end{proof}

We now construct a bijection for ordinary characters.

\begin{pro}
\label{prop:Ordinary}
If $K>1$, then there exists a bijection $\Omega:\Irr_0(B)\to\Irr_0(b)$ such that $\Omega(\chi)(1)$ divides $\chi(1)$ for every $\chi\in\Irr_0(B)$.
\end{pro}

\begin{proof}
Consider an $\norm G D$-transversal $\mathbb{T}$ in the set of $D$-invariant linear characters $\kappa\in\Lin(K)$. 
Using \cite[Lemma 9.3]{N18}, we obtain
\[\Irr_0(B)=\coprod\limits_{\kappa\in\mathbb{T}}\Irr_0(B\mid\kappa)\]
and
\[\Irr_0(b)=\coprod\limits_{\kappa\in\mathbb{T}}\Irr_0(b\mid\kappa).\]
Therefore, it is enough to find a bijection $\Irr_0(B\mid \kappa)\to\Irr_0(b\mid \kappa)$ with the divisibility of character degrees for every $\kappa\in\mathbb{T}$.

Fix $\kappa\in\mathbb{T}$ and set $c:=\bl(\kappa)$. Consider the set $\mathcal{B}$ of blocks $B''\in\Bl(G_\kappa)$ that cover $c$ and such that $(B'')^G=B$. Similarly, define $\mathcal{C}$ to be the set of blocks $b''\in\Bl(\norm{G}{D}_\kappa)$ that cover $c$ and such that $(b'')^{\norm G D}=b$. Notice that $\mathcal{B}$ coincides with the set of blocks containing the Clifford correspondent $\chi_\kappa$ of some $\chi\in\Irr(B)$ over $\kappa$. Moreover, since $\kappa$ is $D$-invariant and $D\in\Syl_p(G)$, it follows that $D$ is a defect group of every block $B''\in\mathcal{B}$. Similarly, $\mathcal{C}$ coincides with the set of blocks containing the Clifford correspondent $\psi_\kappa$ of some $\psi\in\Irr(b)$ over $\kappa$. As before, $D$ is a defect group of every block $b''\in\mathcal{C}$. We claim that induction of block defines a bijection $\mathcal{C}\to\mathcal{B}$, $b''\mapsto (b'')^{G_\kappa}$. First, observe that the map is well defined and, by Brauer's first main theorem, injective. Then, let $B''\in\mathcal{B}$ and consider its Brauer correspondent $b''\in\Bl(\norm{G}{D}_\kappa)$. Since $B''$ covers $c=\bl(\kappa)$ and $\kappa$ is $G_\kappa$-invariant, we deduce that $b''$ also covers $c$. Recalling that $(B'')^G=B$, by the Brauer correspondence we conclude that $(b'')^{\norm{G}{D}}$ coincides with $b$. Therefore $b''\in\mathcal{C}$ and the map is surjective. Now, by the Clifford correspondence we obtain bijections
\begin{equation}
\label{eq:Partition CLifford corr, 1}
\coprod\limits_{B''\in\mathcal{B}}\Irr(B''\mid \kappa)\to\Irr(B\mid \kappa)
\end{equation}
and
\begin{equation}
\label{eq:Partition CLifford corr, 2}
\coprod\limits_{b''\in\mathcal{C}}\Irr(b''\mid \kappa)\to\Irr(b\mid \kappa).
\end{equation}
Since these bijections are given by induction of characters, and recalling that $D\leq G_\kappa$, it follows that \eqref{eq:Partition CLifford corr, 1} and \eqref{eq:Partition CLifford corr, 2} restrict to the sets of irreducible characters of height zero, or equivalently, of $p'$-degree. Now, for all $\chi\in\Irr_0(B)$ there is a unique $B''\in\mathcal{B}$ and a unique $\psi\in\Irr_0(B'')$ such that $\psi^G=\chi$. If there exist bijections $\Omega_{B''}:\Irr_0(B'')\rightarrow \Irr_0(b'')$ such that $\Omega_{B''}(\psi)(1)$ divides $\psi(1)$, then we can define a bijection
\[\Omega:\Irr_0(B)\rightarrow\Irr_0(b)\]
by setting
\[\Omega(\chi):=\Omega_{B''}(\psi)^{\norm G D}\]
and we have $\Omega(\chi)(1)=\Omega_{B''}(\psi)(1)|\norm G D:\norm{G_\kappa}{D}|$ which divides $\psi(1)|G:G_\kappa|=\chi(1)$ by Proposition \ref{gabriel}.
Therefore, in order to conclude, it is no loss of generality to assume that $\kappa$ is $G$-invariant. 

Notice that $\kappa$ has an extension $\wh{\kappa}$ to the inertial subgroup $G_\kappa$ by Lemma \ref{lem 2}. By Gallagher's theorem, for every character $\chi\in\Irr(G\mid \kappa)$, there exists a unique $\eta\in\Irr(G/K)$ such that $\chi=\eta\wh{\kappa}$. Furthermore, if $\eta_1$ and $\eta_2$ lie in the same block, then so do $\eta_1\wh{\kappa}$ and $\eta_2\wh{\kappa}$ (see \cite[Corollary 1.5]{Mur96}). We denote by $\Bl(B,\wh{\kappa})$ the set of blocks $\overline{B}$ of $G/K$ containing a character $\eta$ such that $\eta\wh{\kappa}\in\Irr(B)$. Then
\begin{equation}
\label{eq:Step 1,4}
\Irr_0\left(B\enspace\middle|\enspace \kappa\right)=\coprod\limits_{\overline{B}\in\Bl(B,\wh{\kappa})}\left\lbrace\eta\wh{\kappa}\enspace\middle|\enspace\eta\in \Irr_0\left(\overline{B}\right)\right\rbrace.
\end{equation}
Similarly, if $\Bl\left(b,\wh{\kappa}_{\norm{G}{D}}\right)$ is the set of blocks $\overline{b}$ of $\norm{G}{D}/K$ containing a character $\eta$ such that $\eta\wh{\kappa}_{\norm{G}{D}}\in\Irr(b)$, then we obtain
\begin{equation}
\label{eq:Step 1,5}
\Irr_0\left(b\enspace\middle|\enspace \kappa\right)=\coprod\limits_{\overline{b}\in\Bl(b,\wh{\kappa}_{\norm{G}{D}})}\left\lbrace\sigma\wh{\kappa}_{\norm{G}{D}}\enspace\middle|\enspace\sigma\in \Irr_0\left(\overline{b}\right)\right\rbrace.
\end{equation}
Moreover, it follows by Brauer's first main theorem that induction of blocks yields a bijection $\Bl(b,\wh{\kappa}_{\norm{G}{D}})\to\Bl(B,\wh{\kappa})$ (see the proof of \cite[Corollary 2.5]{Mur98}). Now, as $K>1$, the inductive induction yields a bijection
\[\overline{\Omega}_{\overline{B}}:\Irr_0\left(\overline{B}\right)\to\Irr_0\left(\overline{b}\right)\]
such that $\overline{\Omega}_{\overline{B}}(\overline{\chi})(1)$ divides $\overline{\chi}(1)$ for every $\overline{B}\in\Bl(B,\wh{\kappa})$ with Brauer correspondent $\overline{b}\in\Bl(b,\wh{\kappa}_{\norm{G}{D}})$ and for every $\overline{\chi}\in\Irr_0(\overline{B})$.
Finally, we define 
\[\Omega\left(\chi\right):=\overline{\Omega}_{\overline{B}}\left(\eta\right)\wh{\kappa}_{\norm{G}{D}}\]
for every $\chi\in\Irr(B\mid \kappa)$, $\eta\in\Irr(G/K)$ and $\overline{B}$ such that $\chi=\eta\wh{\kappa}$ and $\bl(\eta)=\overline{B}$. Thanks to \eqref{eq:Step 1,4} and \eqref{eq:Step 1,5}, this defines a bijection between $\Irr_0(B\mid \kappa)$ and $\Irr_0(b\mid \kappa)$. Furthermore, since $\overline{\Omega}_{\overline{B}}(\eta)(1)$ divides $\eta(1)$, it follows that $\Omega(\chi)(1)$ divides $\chi(1)$. This completes the proof.
\end{proof}

Next, we consider the case of Brauer characters.

\begin{pro}
\label{prop:Brauer}
If $K>1$, then there exists a bijection $\Psi:\IBr_0(B)\to\IBr_0(b)$ such that $\Psi(\varphi)(1)$ divides $\varphi(1)$ for every $\varphi\in\IBr_0(B)$.
\end{pro}

\begin{proof}
Define $\overline{G}:=G/K$. Let $\Bl(\overline{G},B)$ be the set of blocks of $\overline{G}$ dominated by $B$ and $\Bl(\overline{\norm G D},b)$ the set of blocks of $\overline{\norm G D}$ dominated by $b$. Observe that $\overline{\norm G D}=\norm{\overline{G}}{\overline{D}}$ and that $\overline{D}$ is a defect group of every block in $\Bl(\overline{\norm G D},b)$ according to \cite[Theorem 9.9]{N98}. Then, block induction defines an injection $\Bl(\overline{\norm G D}, b)\to\Bl(\overline{G})$ by the Brauer's correspondence. Furthermore, this restricts to a bijection $\Bl(\overline{\norm G D}, b)\to\Bl(\overline{G},B)$ according to \cite[Proposition 2.4 (a)]{Nav-Spa14}. Let $\overline{B}\in\Bl(\overline{G},B)$ and consider its Brauer correspondent $\overline{b}\in\Bl(\overline{\norm G D})$. Noticing that $|\overline{G}|<|G|$, by inductive hypothesis we obtain a bijection $\overline{\Psi}_{\overline{B}}:\IBr_0(\overline{B})\to\IBr_0(\overline{b})$ such that $\overline{\Psi}_{\overline{B}}(\overline{\varphi})(1)$ divides $\overline{\varphi}(1)$ for every $\overline{\varphi}\in\IBr_0(\overline{B})$. Since $K\leq \ker\varphi$ for every $\varphi\in\IBr(G)$, we obtain a bijection
\[\Psi:\IBr_0(B)\to\IBr_0(b)\]
by setting
\[\Psi(\overline{\varphi}):=\overline{\Psi}_{\overline{B}}(\overline{\varphi})\]
for every $\overline{B}\in\Bl(\overline{G},B)$, every $\overline{\varphi}\in\IBr_0(\overline{B})$ and where we identify $\overline{\varphi}$ and $\overline{\Psi}_{\overline{B}}(\varphi)$ via inflation with the corresponding Brauer characters of $G$ and $\norm G D$ respectively. Since inflation of characters does not affect character degrees, we conclude that $\Psi(\varphi)(1)$ divides $\varphi(1)$ for every $\varphi\in\IBr_0(B)$.
\end{proof}

As an immediate consequence of Proposition \ref{prop:Ordinary} and Proposition \ref{prop:Brauer} we deduce that it is no loss of generality to assume $K=1$.

\begin{cor}
\label{cor:K=1}
We may assume that $\oh{p}{G}=1$.
\end{cor}

We are now ready to prove Theorem \ref{thm:AM with divisibility}.

\begin{proof}[Proof of Theorem \ref{thm:AM with divisibility}]
Let $N:=\oh{p',p}{G}$ and $L:=\oh{p'}{G}$. By Lemma \ref{lem 1} we know that $D\in\Syl_p(G)$ and that there exists some $G$-invariant character $\vartheta$ of $L$ whose block is covered by $B$. In particular, we deduce that $Q:=N\cap D$ is a Sylow $p$-subgroup of $N$ and $N=LQ$. If $Q\unlhd G$, then Corollary \ref{cor:K=1} implies that $Q=1$ and hence $G$ is a $p'$-group, in which case the result follows trivially. Therefore, we may assume that $\norm G Q < G$.

Observe that $\norm G D\leq \norm G Q$ and consider the block $C:=b^{\norm G Q}$. The inductive hypothesis yields bijections
\[\Omega_C:\Irr_0(C)\to \Irr_0(b)\]
and
\[\Psi_C:\IBr_0(C)\to\IBr_0(b)\]
such that $\Omega_C(\psi)(1)$ divides $\psi(1)$ and $\Psi_C(\varrho)(1)$ divides $\varrho(1)$ for every $\psi\in\Irr_0(C)$ and every $\varrho\in\IBr_0(C)$. Next, recall that $\Irr_0(B)=\Irr_{p'}(G\mid \vartheta)$ and that $\IBr_0(B)=\IBr_{p'}(G\mid \vartheta)$ according to Lemma \ref{lem 1}. Furthermore, \cite[Theorem 0.29]{manzwolf} implies that $\Irr_0(C)=\Irr_{p'}(\norm G Q\mid f_Q(\vartheta))$ and that $\IBr_0(C)=\Irr_{p'}(\norm G Q\mid f_Q(\vartheta))$. Then, applying Proposition \ref{prop:Above Glauberman with divisibility} and Remark \ref{rem} we obtain bijections
\[\Omega_Q:\Irr_0(B)\to\Irr_0(C)\]
and
\[\Psi_Q:\IBr_0(B)\to\IBr_0(C)\]
such that $\Omega_Q(\chi)(1)$ divides $\chi(1)$ and $\Psi_Q(\varphi)(1)$ divides $\varphi(1)$ for every $\chi\in\Irr_0(B)$ and $\varphi\in\IBr_0(B)$. Now, the result follows by setting $\Omega:=\Omega_C\circ\Omega_Q$ and $\Psi:=\Psi_C\circ\Psi_Q$.
\end{proof}

We end this section with a final remark. The bijections given by Proposition \ref{prop:Above Glauberman with divisibility} and Proposition \ref{fongreynolds} can be shown to be compatible with decomposition numbers. As a consequence, it is natural to ask whether the same can be said for the bijections obtained in Theorem \ref{thm:AM with divisibility}. More precisely, we ask the following question.

\begin{que}
\label{question}
Suppose that $G$ is a $p$-solvable groups and let $B$ be a $p$-block of $G$ with Brauer correspondent $b$. Are there bijections $\Omega:\Irr_0(B)\to\Irr_0(b)$ and $\Psi:\IBr_0(B)\to\IBr_0(b)$ such that
\[d_{\chi\varphi}=d_{\Omega(\chi)\Psi(\varphi)}\]
for every $\chi\in\Irr_0(B)$ and every $\varphi\in\IBr_0(B)$?
\end{que}

Although not considered directly in this paper, we think that the above question is worth pointing out and will be assessed in a future investigation.

\section{Proof of Theorem \ref{thm:Dimension divisibility}}
\label{sec:Thm B}
Recall that by \cite[Theorem 3.14]{N98}, if $B$ is a block of a finite group $G$, then $\dim(B)=\sum_{\chi\in\Irr(B)}\chi(1)^2$. Using this equality, together with Proposition \ref{fongreynolds} and Proposition \ref{gabriel}, we can prove Theorem \ref{thm:Dimension divisibility}.

\begin{proof}[Proof of Theorem \ref{thm:Dimension divisibility}]
We argue by induction on the order of $G$. Let $U$ and $\nu\in\Irr(\oh{p'}{U})$ be as in Proposition \ref{fongreynolds} and observe that
\begin{equation}
\label{eq:Proof B}
\dim(B)=\sum_{\chi\in\Irr(B)}\chi(1)^2=\sum_{\psi\in\Irr(U\mid\nu)}\psi^G(1)^2=\sum_{\psi\in\Irr(U\mid\nu)}|G:U|^2\psi(1)^2.
\end{equation}
Assume first that $U<G$. By the inductive hypothesis, and using the fact that $\Irr(U\mid\nu)$ is a full block of $B'\in\Bl(U)$ by \cite[Theorem 10.20]{N98}, we deduce that
\[\dim(B')=\sum_{\psi\in\Irr(U\mid\nu)}\psi(1)^2\]
is divisible by
\[\dim(b')=\sum_{\varphi\in\Irr(\norm U D\mid f_D(\nu))}\varphi(1)^2\]
where $b'\in\Bl(\norm U D)$ is the Brauer correspondent of $B'$ and $\Irr(b')=\Irr(\norm U D \mid f_D(\nu))$ according to \cite[Theorem 0.29]{manzwolf}. We conclude that
\begin{align*}
\dim(b)&=\sum_{\varrho\in\Irr(b)}\varrho(1)^2=\sum_{\varphi\in\Irr(\norm G D \mid f_D(\nu))}\varphi^{\norm G D}(1)^2
\\
&=|\norm G D:\norm U D|^2\sum_{\varphi\in\Irr(\norm U D \mid f_D(\nu))}\varphi(1)^2
\end{align*}
divides $\dim(B)$ by \eqref{eq:Proof B} and Proposition \ref{gabriel}.

Thus, we may assume $U=G$. In this case, if $L=\oh{p'}G$, there exists some $G$-invariant character $\vartheta\in\Irr(L)$ such that $\Irr(B)=\Irr(G\mid\vartheta)$ and $\Irr(b)=\Irr(\norm G D\mid f_D(\vartheta))$. Now 
\[|G:L|\vartheta(1)=\vartheta^G(1)=\frac{1}{\vartheta(1)}\sum_{\chi\in\Irr(G\mid\vartheta)}\chi(1)^2\]
and
\[|\norm G D:\cent L D|f_D(\vartheta)(1)=f_D(\vartheta)^{\norm G D}(1)=\frac{1}{f_D(\vartheta)(1)}\sum_{\varphi\in\Irr(\norm G D\mid f_D(\vartheta))}\varphi(1)^2\]
so
$\dim(B)={|G:L|}\vartheta(1)^2$ and $\dim(b)=|\norm G D:\cent L D| f_D(\vartheta)(1)^2$. To conclude, observe that $|\norm G D:\cent L D|$ divides $|G:L|$ since $\cent L D=L\cap\norm G D$ and that $f_D(\vartheta)(1)$ divides $\vartheta(1)$ by the main result of \cite{Geck}.
\end{proof}

Theorem \ref{thm:Dimension divisibility} shows that the Brauer correspondence is compatible with divisibility of block dimensions for $p$-solvable groups. As mentioned in the introduction, this result does not hold for arbitrary finite groups. These kind of questions about the dimension of blocks are usually not easy to answer in full generality. However, something can be said when we consider normal subgroups. We conclude this section by including a nice result that was communicated to us by Kessar and Linckelmann. 

In what follows, we fix a $p$-modular system $(\mathcal{O}, K,k)$ where $\mathcal{O}$ is a complete discrete valuation ring with residue field $k=\mathcal{O}/J(\mathcal{O})$ of characteristic $p$ and field of fractions $K$ of characteristic $0$. We also assume that $K$ and $k$ are splitting fields for the groups considered below. Then, we assume that all ordinary characters are realised over $K$. Given a block $B$ of $G$, we denote by $\rho_B$ the regular character of $B$, that is, $\rho_B:=\sum_\chi \chi(1)\chi$ where $\chi$ runs over the characters belonging to $B$. Equivalently, $\rho_B$ is the character afforded by the left $\mathcal{O}G$-module $\mathcal{O}Ge_B$ where $e_B$ is the central idempotent corresponding to $B$.

\begin{pro}[Kessar--Linckelmann]
\label{prop:Kessar-Linckelmann}
Let $N$ be a normal subgroup of $G$ and consider blocks $B\in\Bl(G)$ and $b\in\Bl(N)$. If $B$ covers $b$, then 
\[(\rho_B)_N=n\sum\limits_{g\in [G/G_b]}(\rho_b)^g\]
for some positive integer $n$ and where $[G/G_b]$ denotes a set of representatives for the left $G_b$-cosets in $G$. In particular, $\dim(b)$ divides $\dim(B)$.
\end{pro}

\begin{proof}
First, we show that it is no loss of generality to assume that $b$ is $G$-invariant. Let $B'$ be the Fong-Reynolds correspondent of $B$ over $b$ and notice that
\begin{align*}
\rho_B &=\sum\limits_{\chi\in\Irr(B)}\chi(1)\chi=\sum\limits_{\psi\in\Irr(B')}|G:G_b|\psi(1)\psi^G
\\
&=|G:G_b|\left(\sum\limits_{\psi\in\Irr(B')}\psi(1)\psi\right)^G=|G:G_b|(\rho_{B'})^G
\end{align*}
according to \cite[Theorem 9.14]{N98}. Then, using the Mackey formula we deduce that
\begin{equation}
\label{eq:Kessar-Linckelmann 0}
(\rho_B)_N=|G:G_b|((\rho_{B'})^G)_N=|G:G_b|\sum\limits_{g\in [G/G_b]}((\rho_{B'})^g)_N.
\end{equation}
Now, if $(\rho_{B'})_N=n\rho_b$ for some positive integer $n$, then $((\rho_{B'})^g)_N=n(\rho_b)^g$ and the result follows from \eqref{eq:Kessar-Linckelmann 0}. Therefore, it is enough to consider the case where $b$ is $G$-invariant.

Now, consider $\mathcal{O}Ge_b$ as a $\mathcal{O}N$-$\mathcal{O}N$-bimodule. If $[G/N]$ denotes a set of representatives for the $N$-cosets in $G$, then we have a decomposition into indecomposable $\mathcal{O}N$-$\mathcal{O}N$-bimodules
\[\mathcal{O}G=\bigoplus\limits_{x\in[G/N]}\mathcal{O}Nx\]
which implies that
\begin{equation}
\label{eq:Kessar-Linckelmann 1}
\mathcal{O}Ge_b=\bigoplus\limits_{x\in[G/N]}\mathcal{O}Ne_bx.
\end{equation}
Then, since $\mathcal{O}Ge_B$ is a direct summand of $\mathcal{O}Ge_b$ by the previous paragraph and recalling that each $\mathcal{O}Ne_bx$ is indecomposable, the Krull--Schmidt theorem \cite[Theorem 4.6.7]{LinI} applied to \eqref{eq:Kessar-Linckelmann 1} implies that there exists a subset $\mathcal{S}\subseteq [G/N]$ such that
\begin{equation}
\label{eq:Kessar-Linckelmann 2}
\mathcal{O}Ge_B\simeq \bigoplus\limits_{x\in\mathcal{S}}\mathcal{O}Ne_bx.
\end{equation}
Noticing that $\mathcal{O}Ne_bx$ is isomorphic to $\mathcal{O}Ne_b$ as a left $\mathcal{O}N$-module, we deduce that each $\mathcal{O}Ne_bx$ affords the same character $\rho_b$. On the other hand, when viewed as a left $\mathcal{O}N$-module, $\mathcal{O}Ge_B$ affords the character $(\rho_B)_N$ and hence $(\rho_B)_N=|\mathcal{S}|\rho_b$ by \eqref{eq:Kessar-Linckelmann 2}. The second part of the result follows since $\dim(B)=\rho_B(1)$ and $\dim(b)=\rho_b(1)$.
\end{proof}

It is interesting to notice that the above result seems to be extremely hard to prove when we only relay on character theory. It would be interesting to know whether a character theoretic proof of the above result can actually be found.

\vspace{1cm}

DEPARTMENT DE MATEM\`ATIQUES, UNIVERSITAT DE VAL\`ENCIA, 46100 BURJASSOT, VAL\`ENCIA, SPAIN.

\textit{Email address:} \href{mailto:josep.m.martinez@uv.es}{josep.m.martinez@uv.es}

DEPARTMENT OF MATHEMATICS, CITY, UNIVERSITY OF LONDON, EC1V 0HB, UNITED KINGDOM.

\textit{Email address:} \href{mailto:damiano.rossi@city.ac.uk}{damiano.rossi@city.ac.uk}

\end{document}